\documentclass[11 pt]{amsart}

\usepackage{amsmath, amssymb, amsthm, amsfonts, amsxtra, stmaryrd}
 \usepackage{amsaddr}
\usepackage{tikz-cd, enumerate, yfonts, fancyhdr}

\input xy
\xyoption{all}

\theoremstyle{plain}
\newtheorem{theorem}{Theorem}[section]

\newtheorem{corollary}[theorem]{Corollary}
\newtheorem{prop}{Proposition}[section]
\newtheorem{lemma}{Lemma}[section]

\newtheorem*{thm 1.3}{Theorem 1.3}

\theoremstyle{definition}

\def\HH{\mathbb{H}}

\def\RR{\mathbb{R}}

\def\ZZ{\mathbb{Z}}

\newcommand{\piLip}{\pi^{\textup{Lip}}}
\newcommand{\dil}{\textup{dil}}
\newcommand{\lsim}{\lesssim}
\newcommand{\gsim}{\gtrsim}
\newcommand{\Lip}{\textup{Lip}}
\newcommand{\Id}{\textup{Id}}

\title{Quantitative nullhomotopy and the Hopf Invariant}
\author{Luis Kumanduri}

\begin{document}

\maketitle

\begin{abstract}
   Let $G: S^{4n-1} \rightarrow S^{2n}$ be a map with nonzero Hopf Invariant. Using the generalized Hopf invariant introduced by Haj\l{}asz, Schikorra and Tyson, we show that any null-homotopy $F: B^{4n} \rightarrow B^{2n+1}$ of $G$ with small $(2n+1)$-dilation must have large $(2n)$-dilation. Conversely, we show that these results are sharp by constructing smooth null-homotopies with arbitrarily small $(2n+1)$-dilation.
\end{abstract}

\section{Introduction}\label{sec:intro}

The $k$-dilation $\dil_k(G)$ of a smooth map $G$ is a generalization of the Lipschitz constant, measuring how much the map stretches or shrinks $k$-dimensional volumes. In this paper we will study the problem of finding extensions of a map with small $k$-dilation, and obstructions that arise from topology.

Maps with small $k$-dilation are closely connected to maps with low-rank derivative. In particular we have that $\textup{rank}(DG) \le k-1$ if and only if $\dil_k(G) = 0$. The first interesting example of a low-rank map is due to Kaufman in \cite{K}, who constructed a surjective map $G: [0,1]^3 \rightarrow [0,1]^2$ with rank $dG \le 1$ almost everywhere.  In \cite{WY}, Wenger-Young constructed rank $n$ extensions of many maps $S^m \rightarrow S^n$ to maps $B^{m+1} \rightarrow B^{n+1}$. In \cite{HST} Haj{\l}asz-Schikorra-Tyson showed that these extensions do not exist for maps with non-trivial Hopf invariant. 

Problems about low-rank extensions are  related to the Lipschitz homotopy groups of the Heisenberg group introduced in \cite{DHLT} by Dejarnette, Haj{\l}asz, Lukyanenko and Tyson. Any Lipschitz map from a Riemannian manifold to the Heisenberg group $\HH^n$ must have rank $\le n$ almost everywhere, and so a necessary condition for a map $S^m \rightarrow \HH^n$ to admit a Lipschitz null-homotopy is that it has a low-rank extension to the ball. Wenger-Young used their construction to show that some surprising elements of the Lipschitz homotopy groups $\piLip(\HH^n)$ were trivial, while Haj{\l}asz-Schikorra-Tyson showed that the groups $\piLip_{4n-1}(\HH^{2n})$ were non-trivial. For most of the groups $\piLip(\HH^n)$, non-triviality is an open question. In \cite{DHLT} and \cite{RW}, $\piLip_n(\HH^n)$ was shown to be non-trivial and in \cite{H}, Haj{\l}asz show that $\piLip_{n+1}(\HH^n)$ for $n \ge 2$ is non-trivial. In \cite{WY} and \cite{WY2}, Wenger-Young showed the groups $\piLip_n(\HH^1)$ for $n \ge 2$ and $\piLip_k(\HH^n)$ for $k < n$ are all trivial. Using different language, Gromov gave an earlier proof of the triviality of $\piLip_k(\HH^n)$ for $k < n$ in \cite{GrCC}. These are the only known examples.  Our hope is that $k$-dilation will become a useful tool for obstructing low-rank extensions and studying $\piLip(\HH^n)$.

Our main result finds an obstruction to the existence of extensions with small $(2n+1)$-dilation.

\begin{theorem}\label{thm:main}
Let $G: S^{4n-1} \rightarrow S^{2n}$ be a smooth map with nonzero Hopf invariant, and let $i: S^{2n} \rightarrow \RR^k$ be a smooth embedding. Let $F$ be a smooth null-homotopy of $i \circ G$ with $2n$-dilation $\dil_{2n}(F) \le C$. Then there is a constant $c(C,G,i) > 0$ so that $\dil_{2n+1}(F) \ge c$. 

\end{theorem}

This result is related to the following theorem of Haj{\l}asz-Schikorra-Tyson (\cite{HST}), which can be used to show that $\piLip_{4n-1}(\HH^{2n})$ is non-trivial.

\begin{theorem}\label{thm:low_rank}
The map $i \circ G$ does not admit a Lipschitz null-homotopy $F$ with $\textup{rank}(DF) \le 2n$ almost everywhere.
\end{theorem}

In order to prove Theorem \ref{thm:low_rank}, Haj{\l}asz-Schikorra-Tyson introduced a generalized Hopf invariant of maps $G: S^{4n-1} \rightarrow \RR^k$ with $\textup{rank}(DG) \le 2n$ and showed that it was invariant under Lipschitz homotopies $F$ with $\textup{rank}(DF) \le 2n$ almost everywhere. The generalized Hopf invariant is similar to the Whitehead definition of the Hopf invariant and the proof of invariance proceeds similarly. However, a Lipschitz map of rank $\le 2n$ almost everywhere may not necessarily be approximated by a smooth one preserving the rank condition; therefore one needs to work through the serious technical details to make this argument rigorous for general Lipschitz functions.

The proof of Theorem \ref{thm:main} uses similar ideas. We will use the generalized Hopf invariant to obstruct homotopies with bounded $2n$-dilation and small $(2n+1)$-dilation. The generalized Hopf invariant is no longer an invariant, but  we show that it is robust under homotopies of small $(2n+1)$-dilation. This gives us the advantage of being able to work directly with smooth functions while also proving a similar theorem. As an aside, it may be possible to use the ideas of Theorem \ref{thm:main} to prove Theorem \ref{thm:low_rank}, but the details are not trivial. A rank $k-1$ Lipschitz map can not necessarily be approximated by maps of small $k$-dilation, instead they are approximated by maps whose $k$-dilation is small in some integral sense. Improving Theorem \ref{thm:main} to obstruct such an integral from being small, if true, would likely require a finer analysis than the tools in this paper seem to give.

It is natural to ask if we can drop the condition on $2n$-dilation, but our next result shows that this is impossible and that Theorem \ref{thm:main} is sharp. 

\begin{theorem}\label{thm:construction}
Let $G: S^m \rightarrow S^n$ be a smooth map with $m > n > \frac{m}{2}$. Then for any $\epsilon > 0$, $G$ extends to a map $F: B^{m+1} \rightarrow B^{n+1}$ with $\dil_{n+1}(F) < \epsilon$.
\end{theorem}

Extending the embedding $i: S^{2n} \rightarrow \RR^k$ to a ball and applying Theorem \ref{thm:construction} immediately gives the following corollary.

\begin{corollary}\label{cor:construction}
$i \circ G$ admits null-homotopies $F: B^{4n} \rightarrow \RR^k$ with arbitrarily small $(2n+1)$-dilation.
\end{corollary}

To prove Theorem \ref{thm:construction}, we will utilize Guth's h-principle for k-dilation proved in \cite{Gu}. Guth showed that maps $S^m \rightarrow S^n$ are homotopic to maps of arbitrarily small $k$-dilation for $k > \frac{m+1}{2}$. Although this theorem doesn't apply directly, we will adapt the construction to create extensions. 

Theorem \ref{thm:construction} is similar to the previously mentioned result of Wenger-Young, who showed in \cite{WY} that if all of the Hopf-Hilton invariants of $G: S^m \rightarrow S^n$ vanish, then $G$ admits a Lipschitz extension $B^{m+1} \rightarrow B^{n+1}$ which has rank $\le n$ almost everywhere, i..e $\dil_{n+1} = 0$. This extension was upgraded to $C^1$ in \cite{GHP}. In particular, if $G$ is homotopic to a suspension, then $G$ admits a rank $n$ extension to the ball. By the Freudenthal Suspension Theorem, when $n+1 \le m < 2n-1$ all $G$ are homotopic to suspensions, so Theorem \ref{thm:construction} gives a new answer for maps $S^{2n-1} \rightarrow S^n$.

 \subsection*{Acknowledgements}
 I would like to thank Larry Guth and Robin Elliott for helpful conversations about this project. This work was supported by the National Science Foundation Graduate Research Fellowship under grant number 1745302.

\section{Preliminaries}\label{sec:background}
\subsection{k-dilation}\label{subsec:k-dil}

In this section we will collect a few basic facts about $k$-dilation that we use throughout this paper. For proofs of these facts and a broader discussion of $k$-dilation in topology, we refer the reader to \cite{Gu}. 

Let $F: M \rightarrow N$ be a $C^1$ map between Riemannian manifolds. The $k$-dilation $\dil_k(F)$ is defined by

\[ \dil_k(F) = \sup_{\Sigma^k \subset M} \frac{\textup{Vol}_k(F(\Sigma^k))}{\textup{Vol}_k(\Sigma^k)}. \]

Here $\Sigma_k$ ranges over $k$-dimensional submanifolds of $M$. In particular, if $F(M)$ lies in a $k-1$-dimensional subset of $N$, then $\dil_k(F) = 0$. $k$-dilation can also be computed infinitesimally from the singular values $s_1(x) \ge s_2(x) \ge \cdots$ of $DF_x: T_xM \rightarrow T_{f(x)}N$ by

\[ \dil_k(F) = \sup_{x \in M} \prod_{i=1}^k s_i(x). \]

The latter perspective shows that if $dF$ has rank $\le k-1$, then $\dil_{k}(F) = 0$. This definition also allows us to define a local notion of the $k$-dilation, by the function

\[ \dil_k(F)(x) = \prod_{i=1}^k s_i(x). \]

In addition, we can easily relate the $k$-dilation for different values of $k$.

\begin{prop}\label{prop:k-relation}
If $k > j$, the $\dil_j(F)^{1/j} \ge \dil_k(F)^{1/k}$
\end{prop}

We use this proposition several times in the paper, but of particular interest to us is the case of $j = 1$. $\dil_1(F) = \Lip(F)$, and so we get the inequality $\Lip(F)^k \ge \dil_k(F)$. In particular, if a map has small Lipschitz constant, it will also have small $k$-dilation.

To prove Theorem \ref{thm:main}, we will take a dual perspective on  $k$-dilation, working with differential forms instead of submanifolds. 

\begin{prop}\label{prop:forms}
Let $\omega \in \Omega^k(N)$ be a differential form, then

\[ ||F^\ast \omega||_{L^\infty(\Omega^k(M))} \le \dil_k(F)||\omega||_{L^\infty(\Omega^k(N))}. \]
\end{prop}

Proposition \ref{prop:forms} is very useful for establishing obstructions to small $k$-dilation in special cases. Many rational homotopy invariants can be computed using the pullback of differential forms, and Proposition \ref{prop:forms} allows one to bound these invariants using $k$-dilation.

\subsection{Generalized Hopf Invariant}\label{subsec:gen_hopf}

In \cite{HST}, Haj{\l}asz-Shikorra-Tyson introduced the generalized Hopf invariant of a rank $2n$-map $F: S^{4n-1} \rightarrow \RR^k$. By rank $2n$ we mean that $\textup{rank}(DF) \le 2n$. In this section we will give the definition and discuss some properties and results about the generalized Hopf invariant. 

Let $\omega \in \Omega^{2n}(\RR^k)$ be a $2n$-form. Then $F^\ast \omega$ is a closed form on $S^{4n-1}$. This is because $dF^\ast \omega = F^\ast d\omega = 0$, where the last fact follows since $F$ has rank $2n$ and $d\omega$ is a $(2n+1)$-form. Thus there is a form $\alpha$ on $S^{4n-1}$ so that $d\alpha = F^\ast \omega$, and we can define the generalized Hopf invariant $H_\omega(F)$ by

\[ H_\omega(F) = \int_{S^{4n-1}} \alpha \wedge F^\ast \omega .\]

\begin{prop}\label{prop:independence}
$H_\omega(F)$ is independent of the choice of primitive $\alpha$.
\end{prop}
\begin{proof}
If $\beta$ is another primitive of $F^\ast \omega$, then $d(\alpha - \beta) = 0$. Therefore we can find $\gamma$ so that $d\gamma = (\alpha-\beta)$. But then we have that $d(\gamma \wedge F^\ast \omega) = \alpha \wedge F^\ast \omega - \beta \wedge F^\ast \omega$, so Stokes' Theorem gives the result.
\end{proof}

We can now consider the case that $F$ factors through an embedding $i: S^{2n} \rightarrow \RR^k$ via a map $G: S^{4n-1} \rightarrow S^{2n}$, and $\int_{S^{2n}} i^\ast \omega = 1$ Then the definition of $H_\omega$ coincides with the Whitehead definition of the Hopf invariant (\cite{BT}) and we see that $\textup{Hopf}(G) = H_\omega(i \circ G)$.

Lastly we state the main result of \cite{HST}, and sketch the proof ignoring the technical details needed to make the argument work for general Lipschitz homotopies.

\begin{theorem}\label{thm:HST_Main}
Let $F$ be a Lipschitz homotopy between rank $2n$ maps $F_0$ and $F_1$ such that $\textup{rank}(DF) \le 2n$ almost everywhere. Then $H_\omega(F_0) = H_\omega(F_1)$
\end{theorem}
\begin{proof}[Proof Sketch]
We will give the proof pretending that $F$ is smooth. In this case, $F^\ast \omega$ is a closed $2n$-form on $S^{4n-1} \times I$, hence exact and we can find $\alpha$ on $S^{4n-1} \times I$ so that $d\alpha = F^\ast \omega$. Notice that $F^\ast \omega|_{S^{4n-1} \times \{i\}} = F_i^\ast \omega$, and so by Stokes' Theorem.

\[ H_\omega(F_1) - H_\omega(F_0) = \int_{S^{4n-1} \times I} d(\alpha \wedge F^\ast \omega). \]

where $d(\alpha \wedge F^\ast \omega) = F^\ast (\omega \wedge \omega) - \alpha \wedge F^\ast d\omega $. But since $F$ has rank $2n$, both $F^\ast d\omega = 0$ and $F^\ast (\omega \wedge \omega) = 0$
\end{proof}

\section{Proof of Theorem \ref{thm:main}}\label{sec:main_proof}

In this section we prove Theorem \ref{thm:main}, which will follow from a more technical result about the generalized Hopf invariant.

\begin{theorem}\label{thm:main_inequality}
Let $F_0,F_1: S^{4n-1} \rightarrow \RR^k$ be smooth maps of rank $2n$, and $\omega \in \Omega^{2n}(\RR^k)$ be a $2n$-form, with $||\omega||_{L^\infty},||d\omega||_{L^\infty} \le c$. Let $F$ be a smooth homotopy between $F_0$ and $F_1$. Then there is a constant $c'$ depending only on $n$ and $c$ such that

\[ |H_\omega(F_0) - H_\omega(F_1)| \le c'(\dil_{2n}(F)\dil_{2n+1}(F) + \dil_{2n+1}(F)^2+\dil_{2n+1}(F)^{\frac{4n}{2n+1}}). \]

Recall that $H_\omega$ is the generalized Hopf invariant discussed in Section \ref{subsec:gen_hopf}.

\end{theorem}

In other words, Theorem \ref{thm:main_inequality} tells us that the generalized Hopf invariant is robust under homotopies of small $(2n+1)$-dilation, even though it is not an invariant for general homotopies. In particular, if $F$ is $L$-Lipschitz, then $\dil_{2n}(F) \le L^{2n}$, and we obtain an inequality of the form

\[ |H_\omega(F_0) - H_\omega(F_1)| \lsim \dil_{2n+1}(F)+\dil_{2n+1}(F)^2. \]

Our strategy is to mimic the sketch we gave of Theorem \ref{thm:HST_Main}. The proof relies on the form $F^\ast \omega$ being closed. In our case, it is no longer closed but instead is almost closed in a precise sense. We will then use the co-isoperimetric inequality to show that $F^\ast \omega$ is almost exact. Using these estimates, we will bound the failure of Theorem \ref{thm:HST_Main} in terms of $\dil_{2n}(F)$ and $\dil_{2n+1}(F)$, which will give the result.

\begin{proof}[Proof of Theorem \ref{thm:main_inequality}]
In the proof we will use the notation $A \lsim B$ to mean there is  a constant $c'$ depending on $n$ and $c$ so that $A \le c'B$. We will use the following lemma repeatedly.

\begin{prop}[Co-Isoperimetric Inequality (\cite{Gr})]\label{lem:coiso}
Let $M = S^{4n-1},S^{4n-1} \times I$ and $\alpha \in \Omega^\ast(M)$ be an exact form. Then $\alpha = d\beta$ with  $||\beta||_{L^\infty} \lsim ||\alpha||_{L^\infty}$
\end{prop}

Proposition \ref{lem:coiso} follows from Lemma 7.13 of \cite{Gr}. Other variants of this lemma appear in \cite{CDMW} and \cite{Gu2}, and the result is well-known as a dual to the linear isoperimetric inequality on a compact manifold.

Consider the $2n$-form $F^\ast \omega$, and notice that $||F^\ast \omega||_{L^\infty} \lsim \dil_{2n}(F)$. $F^\ast \omega$ is not closed, but $dF^\ast\omega = F^\ast d\omega$, so $ ||dF^\ast\omega|| \lsim \dil_{2n+1}(F) $. By the co-isoperimetric inequality, there is a $2n$-form $\alpha$ on $S^{4n-1} \times I$ so that $d\alpha = F^\ast d\omega$ and $ ||\alpha||_{L^\infty} \lsim \dil_{2n+1}(F) $. 
$F^\ast\omega - \alpha$ is  a closed $2n$-form, and hence exact, so by the co-isoperimetric inequality we can find $\beta \in \Omega^{2n-1}(S^{4n-1} \times I)$ with $d\beta = F^\ast \omega-\alpha$ and $ ||\beta||_{L^\infty} \lsim \dil_{2n}(F)+\dil_{2n+1}(F) $.

By Stokes' Theorem, we have that

\begin{equation}\label{eq:Stoke} \int_{S^{4n-1} \times \{1\}} \beta \wedge F^\ast \omega -  \int_{S^{4n-1} \times \{0\}} \beta \wedge F^\ast \omega = \int_{S^{4n-1} \times [0,1]} d(\beta \wedge F^\ast \omega). \end{equation}

Our goal is to relate the boundary terms to $H_\omega(F_i)$, with an error term involving $\dil_{2n+1}(F)$, and then expand $d(\beta \wedge F^\ast \omega)$ and bound all the terms individually. 

On the spheres $S^{4n-1} \times \{i\}$, for $i = 0,1$ we have that $dF^\ast\omega = F_i^\ast d\omega = 0$ since the maps $F_i$ have rank $\le 2n$. Therefore there is a form $\tilde{\beta}$ on $S^{4n-1}$ with $d\tilde{\beta} = F_i^\ast\omega$, and $||\tilde{\beta}|| \lsim \dil_{2n}(F)$. Thus $\int_{S^{4n-1} \times \{ i \}} \tilde{\beta}\wedge F^\ast\omega = H_w(F_i)$ for $i = 0,1$, so we have that

\begin{equation}\label{eq:endpoint_one} \int_{S^{4n-1} \times \{i\}} \beta \wedge F^\ast \omega = H_w(F_i) + \int_{S^{4n-1} \times \{i\}} (\beta-\tilde{\beta}) \wedge F^\ast \omega. \end{equation}

Now $d(\beta-\tilde{\beta}) = \alpha|_{S^{4n-1}}$, so there is a form $\gamma \in \Omega^{2n-1}(S^{4n-1})$ such that $d\gamma = \alpha$ and $||\gamma|| \lsim \dil_{2n+1}(F)$. $\beta-\tilde{\beta}-\gamma$ is a closed and exact $2n-1$-form. Thus we can find $\eta$ so that $d\eta = \beta - \tilde{\beta}-\gamma$. Then $d(\eta \wedge F^\ast \omega) = (\beta-\tilde{\beta}) \wedge F^\ast \omega - \gamma \wedge F^\ast\omega$ since $dF^\ast \omega|_{S^{4n-1} \times \{i\}} = 0$. Therefore by Stokes' Theorem 

\begin{align}\label{eq:endpoint_two} \bigl|\int_{S^{4n-1} \times \{i\}} (\beta-\tilde{\beta}) \wedge F^\ast \omega\bigr| &= \bigl|\int_{S^{4n-1} \times \{i\}} \gamma \wedge F^\ast\omega\bigr| \\
&\lsim \dil_{2n}(F)\dil_{2n+1}(F) \nonumber \end{align}

where the latter estimate follows from our estimates for $\gamma$ and $F^\ast \omega$. Combining (\ref{eq:Stoke}), (\ref{eq:endpoint_one}) and (\ref{eq:endpoint_two}) we obtain that

\begin{equation}\label{eq:Hopf_inequality} |H_\omega(F_1) - H_\omega(F_0)| \lsim \bigl|\int_{S^{4n-1}\times [0,1]} d(\beta \wedge F^\ast \omega)\bigr| + \dil_{2n}(F)\dil_{2n+1}(F). \end{equation}

It remains to estimate $d(\beta \wedge F^\ast \omega) = -\beta \wedge F^\ast d\omega + F^\ast(\omega \wedge \omega) - \alpha \wedge F^\ast \omega$. We have that $||\alpha \wedge F^\ast \omega|| \lsim ||\alpha||||F^\ast \omega|| \lsim (\dil_{2n}(F))(\dil_{2n+1}(F))$.  Similarly we can estimate $||F^\ast(\omega \wedge \omega)|| \lsim \dil_{4n}(F) \le \dil_{2n+1}(F)^{\frac{4n}{2n+1}}$. Finally we can estimate $||\beta \wedge F^\ast d\omega|| \lsim (\dil_{2n}(F)+\dil_{2n+1}(F))\dil_{2n+1}(F)$. Integrating all of these estimates over $S^{4n-1} \times [0,1]$ proves the theorem.

\end{proof}

We can now prove our main result

\begin{proof}[Proof of Theorem \ref{thm:main}]

We use the notation in the statement of Theorem \ref{thm:main}. In order to apply Theorem \ref{thm:main_inequality}, take $F_0 = i \circ G$, and $F_1$ to be any constant map. We can choose $\omega$ to be any extension of a volume form on $i(S^{2n})$ which vanishes outside some compact set, so that there exists $c$ depending only on $i$, with $||\omega||,||d\omega|| \le c$. Then $H_\omega(F_0) = \textup{Hopf}(G)$ and $H_\omega(F_1) = 0$. If there was a null-homotopy $F$ with $\dil_{2n}(F) \le C$ and $\dil_{2n+1}(F) < \epsilon$, then Theorem \ref{thm:main_inequality} implies

\[ |\textup{Hopf}(G)| \lsim C\epsilon+\epsilon^2+\epsilon^{\frac{4n}{2n+1}}. \]

But since $\textup{Hopf}(G)$ is nonzero, we cannot take $\epsilon$ arbitrarily small, which gives the result.

\end{proof}






\section{Proof of Theorem \ref{thm:construction}}\label{sec:construction}

For convenience, we restate Theorem \ref{thm:construction}

\begin{thm 1.3}
Let $G: S^m \rightarrow S^n$ be a smooth map with $m > n > \frac{m}{2}$. Then for any $\epsilon > 0$, $G$ extends to a map $F: B^{m+1} \rightarrow B^{n+1}$ with $\dil_{n+1}(F) < \epsilon$.
\end{thm 1.3}

Our strategy is to modify Guth's construction of maps with small $k$-dilation in \cite{Gu}, which we will now briefly describe. Guth started with a map $F_0: S^m \rightarrow S^n$, and constructed degree $1$-maps $\Phi: S^m \rightarrow S^m$ and $\Psi: S^n \rightarrow S^n$, so that $\Psi \circ F_0 \circ \Phi$ has small $k$-dilation.

$\Psi$ was constructed by splitting $S^n = A \cup B$ into two subsets. $\Psi(A)$ lies in a $(k-1)$-dimensional subcomplex of $S^n$, and hence $\dil_k(\Psi|_{A}) = 0$, while $\Psi|_{B}$ has bounded Lipschitz constant.

To construct $\Phi$, we only have to worry about $\dil_k(\Phi|_{\Phi^{-1}(F_0^{-1}(B))})$. Guth did so by constructing embeddings $F_0^{-1}(B) \rightarrow S^m$ which expanded all directions by a large factor. When $k$ is in the right range of dimensions, these embeddings are isotopic to the inclusion $F_0^{-1}(B) \hookrightarrow S^m$, and so extend to a diffeomorphism $\Phi^{-1}: S^m \rightarrow S^m$. Because these embedding expand all directions, $\Phi$ has small Lipschitz constant on the desired subset and therefore small $k$-dilation.

To prove Theorem \ref{thm:construction}, we will have to modify this argument to work so that $\Psi$ and $\Phi$ are now maps from $B^{n+1} \rightarrow B^{n+1}$ and $B^{m+1} \rightarrow B^{m+1}$ which fix the boundary of the balls, and therefore the resulting map is an extension.

\begin{proof}[Proof of Theorem \ref{thm:construction}]

For the following proof, we will take $B^{m+1},B^{n+1}$ to be unit balls of radius $1$, and equip $S^m$ and $S^n$ with the round metric. 

Fix a smooth extension $F_0: B^{m+1} \rightarrow B^{n+1}$ of $G$.  We will use the notation $A \lsim B$ to mean that there is a constant $c(F_0)$ depending only on $F_0$ so that $A \le cB$. Fix small constants $\delta > 0$ and $W \gsim 1$. $W$ will be chosen later depending only on $F_0$, and $\delta$ will be as small as we like and eventually taken to $0$. We will find maps $\Phi: B^{m+1} \rightarrow B^{m+1}$ and $\Psi: B^{n+1} \rightarrow B^{n+1}$ such that $\Phi|_{S^m},\Psi|_{S^n} = \Id$ and $\dil_{n+1}(\Psi \circ F_0 \circ \Phi) < \epsilon$ as $\delta \rightarrow 0$.

Without loss of generality the origin $O$ is a regular value of $F_0$. There is a radius $\frac{1}{2} \lsim r < 1$ so that there is a diffeomorphism $F_0^{-1}(B_r(O)) \cong F_0^{-1}(O) \times B_r(O)$ and the map $F_0|_{F_0^{-1}(B_r(O))}$ is given by projection onto the second factor.

Let $Q = \delta \ZZ^{n+1} \cap B_r(O)$. Define $V_W$ as the $W\delta$ neighborhood of $Q \in B_r(O)$. Notice that the complement $B_r(O)/V_W$ retracts relative to the boundary onto an $n$-dimensional complex, namely the union of the sphere $\partial B_r(O)$ with the $n$-dimensional lattice that is dual to $Q$. The map $\Psi$ will essentially be given by this retraction. 

\begin{lemma}\label{lem:squeeze}
There is a map $\Psi: B^{n+1} \rightarrow B^{n+1}$ with the following properties.
\begin{enumerate}
    \item $\Lip(\Psi) \lsim 1$
    \item $\dil_{n+1}(\Psi|_{B^{n+1}/V_W}) = 0$
    \item $\Psi|_{S^n} = \Id$
\end{enumerate}
\end{lemma}
\begin{proof}[Proof of Lemma \ref{lem:squeeze}]

Let $\tilde{\Sigma}^n$ be the $n$-dimensional lattice dual to $\delta\ZZ^{n+1} \subset \RR^{n+1}$. That is to say if $\Sigma$ is the unit cubical lattice, then $\tilde{\Sigma} = \delta\Sigma + (\frac{\delta}{2},\frac{\delta}{2},\ldots,\frac{\delta}{2})$. We will use the following construction from \cite{Gu}.

\begin{lemma}[Lemma 11.5 of \cite{Gu}]\label{lem:squeeze_guth}
Let $W > 0$ be any constant. Then there is a $\delta\ZZ^{n+1}$-periodic map $R: \RR^{n+1} \rightarrow \RR^{n+1}$ with the following properties.
\begin{enumerate}
    \item $R$ maps the complement of the $W\delta$ neighborhood of $\delta\ZZ^{n+1}$ onto $\tilde{\Sigma}$
    \item For any $y \in \RR^{n+1}$, $|R(y)-y| \lsim \delta$.
    \item $\Lip(R) \le c(W)$, in particular the Lipschitz constant depends only on $W$ and is independent of $\delta$.
\end{enumerate}
\end{lemma}

Using the map $R$, we will be able to construct $\Psi$. Let $\lambda: [0,1] \rightarrow [0,1]$ be a smooth increasing map such that 

\[ \lambda(x) = \begin{cases} \frac{x}{r} & x \le \frac{r}{4} \\
1 & x \ge \frac{r}{2}
\end{cases} \]

Fix $\lambda$ so that $\Lip(\lambda) \lsim 1$, and consider the smooth map $\Lambda: B^{n+1} \rightarrow B^{n+1}$ defined by $\Lambda(0) = 0$ and otherwise $\Lambda(x) = \lambda(||x||)\frac{x}{||x||}$. Notice that $\Lambda$ maps everything outside $B_{\frac{r}{2}}(O)$ to $S^n$, and that $\Lip(\Lambda) \lsim 1$. On concentric spheres about the origin of radius $r' \ge r/2$, $\Lambda$ is the dilation map $x \rightarrow \frac{x}{r'}$. In particular $\Lambda$ is the identity map on the boundary $S^n$.

By property (2) of the map $R$, the map $S^n \rightarrow S^n$ given by $\Lambda \circ R|_{\partial B_{\frac{3r}{4}}(O)}$ is homotopic to $\Lambda|_{\partial B_{\frac{3r}{4}}(O)}$ by a smooth homotopy $H: S^n \times [0,1] \rightarrow S^n$ with $\Lip(H) \lsim 1$. This is because $R$ moves points at most $c(F_0)\delta$ for some constant $c$, so $\Lambda \circ R|_{\partial B_{\frac{3r}{4}}(O)}$ moves points by at most a constant multiple of $\frac{\delta}{r}$ when viewed as a map $S^n \rightarrow S^n$.

We can now define $\Psi: B^{n+1} \rightarrow B^{n+1}$ as follows.

\[ \Psi(x) = \begin{cases} \Lambda(R(x)) & x \in B_{\frac{3r}{4}}(O) \\
H(\frac{3rx}{4||x||},\frac{4||x||}{r}-3) & x \in B_r(O)/B_{\frac{3r}{4}}(O) \\
\Lambda(x) & x \in B^{n+1}/B_r(O)
\end{cases}. \]

By modifying $H$ near $0$ and $1$, $\Psi$ can be made smooth. Since $H$, $\Lambda$, $R$ are all maps with $\Lip \lsim 1$, and $r \gsim 1$, we immediately obtain that $\Lip(\Psi) \lsim \frac{1}{r} \lsim 1$. Moreover, by the definition of $\Lambda$ and property (1) of $R$, $\Psi({B^{n+1}/V_W})$ lies in $\Lambda(\tilde{\Sigma}) \cup S^n$, which is an $n$-dimensional subcomplex of $B^{n+1}$. In particular we have that $\dil_{n+1}(\Psi|_{B^{n+1}/V_W}) = 0$. Finally we have that $\Psi|_{S^n} = \Lambda|_{S^n} = \Id$, as desired.

\end{proof}

With the construction of $\Psi$ in hand, let us turn our focus to the construction of $\Phi$. Let $U_W = F_0^{-1}(V_W)$.  Using the diffeomorphism $F_0^{-1}(B_r(O)) \cong F_0^{-1}(O) \times B_r$, we see that $U_W \cong F_0^{-1}(O) \times V_W$, and the map $F_0$ is given by projection onto the second factor. Let $h_0$, $h_1$ be the metrics on $F_0^{-1}(O)$, $V_W$ respectively given by restricting the standard metrics from the balls. Let $g_0 = h_0 + h_1$ be a metric on $U_W$ under the product identification. Note that $g_0$ is bi-Lipschitz to the restriction of the standard metric on $U_W$ with bi-Lipschitz constant depending only on $F_0$. Let $g_1 = \delta^2h_0 + h_1$ be a metric on $U_W$ given by shrinking the $F_0^{-1}(O)$ factor. Notice that $\Lip(F_0: (U_W,g_1) \rightarrow V_W) \lsim 1$, since all we did is shrink the fibers of $F_0$.

\begin{lemma}\label{lem:phi_construction}
We can choose $W \gsim 1$ so that there is a diffeomorphism $\Phi: B^{m+1} \rightarrow B^{m+1}$ with the following properties
\begin{enumerate}
    \item $\Phi|_{S^m} = \Id$
    \item $\Lip(\Phi: \Phi^{-1}(U_W)\rightarrow (U_W,g_1)) \lsim \delta^{\frac{m-n}{n+1}}$. Here $\Phi^{-1}(U_W)$ is equipped with the metric given by restricting the metric on $B^{m+1}$
\end{enumerate}
\end{lemma}
\begin{proof}[Proof of Lemma \ref{lem:phi_construction}]
Note that $U_W$ can equally be viewed as the preimage of a map $S^{m+1} \rightarrow S^{n+1}$ given by mapping $F_0$ and collapsing the boundaries to a point, so Guth's quantitative embedding applies.

\begin{lemma}[Lemma 11.3 of \cite{Gu}]\label{lem:embedding_guth}
If $n > \frac{m}{2}$, there exists $W > 0$ independent of $\delta$ such that there is an embedding $I: (U_W,g_1) \rightarrow S^{m+1}$ isotopic to the inclusion $U_W \rightarrow S^{m+1}$ such that $I$ increases all lengths by a factor $L \gsim \delta^{\frac{n-m}{n+1}}$
\end{lemma}

This isotopy can be taken to miss a point $p$. When $\delta$ is small, $F_{0}^{-1}(B_r(O))$ lies in $(1-\delta)B^{m+1}$, and there is a $c$-expanding map $P: S^{m+1}/p \rightarrow B^{m+1}(1-\delta)$ for $c \gsim 1$. Then $P \circ I$ is an embedding isotopic to the inclusion which expands all lengths by a factor of $cL \gsim \delta^{\frac{n-m}{n+1}}$.
 Since this isotopy lies in $(1-\delta)B^{m+1}$, we can extend it to a diffeomorphism $\Phi^{-1}$ of $B^{m+1}$ which is the identity near the boundary. Since $\Phi^{-1}$ is expanding on $U_W$, we immediately obtain that $\Phi|_{\Phi^{-1}(U_W)}$ has small Lipschitz constant, as desired.
\end{proof}

We can now complete the proof of Theorem \ref{thm:construction}. Let $F = \Psi \circ F_0 \circ \Phi$. Let $A_1 = \Phi^{-1}(U_W)$, and $A_2 = B^{m+1}/A_1$. By construction, $F_0 \circ \Phi(A_2) \subset B^{n+1}/V_W$, and so by Lemma \ref{lem:squeeze}, we have that $\dil_{n+1}(F|_{A_2}) = 0$. On $A_1$, we have that

\[ \Lip(F|_{A_1}) \le \Lip(\Psi)\Lip(F_0)\Lip(\Phi|_{A_1}) \lsim \delta^{\frac{m-n}{n+1}}, \] which follows from Lemmas \ref{lem:squeeze} and \ref{lem:phi_construction} as $\Lip(F_0: (U_W,g_1) \rightarrow V_W \lsim 1$. Then by Proposition \ref{prop:k-relation}, we have $\dil_{n+1}(F|_{A_1}) \le \Lip(F|_{A_1})^{n+1} \lsim \delta^{m-n}$. Thus we have that $\dil_{n+1}(F) \lsim \delta^{m-n}$, and taking $\delta \rightarrow 0$ we can make the $(n+1)$-dilation as small as we like. Lastly, since $\Phi|_{S^m} = \Id$ and $\Psi|_{S^n} = \Id$, we see that $F$ agrees with $F_0$ on $S^m$, and is therefore an extension of $G$ to the ball.

\end{proof}






\bibliographystyle{plain}
\bibliography{main}

\end{document}